\documentclass[10pt]{amsart}
\usepackage{amssymb,amsthm,amsmath,amsfonts}
\usepackage{hyperref}

\makeatletter
\g@addto@macro\th@plain{\thm@headpunct{}}
\makeatother

\newtheorem{thm}{Theorem}[section]

\newtheorem{lem}[thm]{Lemma}
\newtheorem{prop}[thm]{Proposition}

\newtheorem{mainthm}[thm]{Main Theorem}
\theoremstyle{definition}

\numberwithin{equation}{section}

\newcommand{\xx}{ {\textbf x} }
\newcommand{\ab}{ {\textbf a} }
\newcommand{\cc}{ {\textbf c} }
\newcommand{\yy}{ {\textbf y} }
\newcommand{\ttt}{ {\textbf t} }
\newcommand{\zz}{ {\textbf z} }
\newcommand{\ee}{ {\textbf e} }

\newcommand{\ub}{ {\textbf u} }
\newcommand{\vb}{ {\textbf v} }

\newcommand{\pb}{ {\textbf p} }
\newcommand{\pbo}{ {\textbf p}^\bot }

\newcommand{\qb}{ {\textbf q} }
\newcommand{\qbo}{ {\textbf q}^\bot }

\newcommand{\VV}{ \mathcal{V} }
\newcommand{\RR}{\mathbb{R}}
\newcommand{\KK}{\mathbb{K}}

\newcommand{\LL}{\mathbb{L}}
\newcommand{\PP}{\mathbb{P}}
\newcommand{\En}{\mathbb{E}}

\newcommand{\DD}{\mathcal{D}}

\newcommand{\tr}{\mathrm{tr}\,}
\newcommand{\Trace}{\mathrm{Trace}\,}
\newcommand{\DDet}{\mathrm{Det}}

\providecommand{\scalar}[1]{\left\langle#1\right\rangle}

\title{The Lukacs-Olkin-Rubin theorem on symmetric cones through Gleason's theorem}
\author[B. Ko\l{}odziejek]{Bartosz Ko\l{}odziejek}
\address{Faculty of Mathematics and Information Science\\Warsaw University of Technology\\Pl. Politechniki 1\\00-661 Warszawa, Poland}
\email{kolodziejekb@mini.pw.edu.pl}

\subjclass[2010]{Primary 62H05.}

\keywords{Wishart distribution, symmetric cones, independence, Gleason's theorem, functional equations}

\begin{document}

\begin{abstract}
We prove the Lukacs characterization of the Wishart distribution on non-octonion symmetric cones of rank greater than $2$. We weaken the smoothness assumptions in the version of the Lukacs theorem of [Bobecka–-Weso{\l}owski, Studia Math. 152 (2002),
147–-160]. The main tool is a new solution of the Olkin-Baker functional equation on symmetric cones, under the assumption of continuity of respective functions. It was possible thanks to the use of Gleason's theorem.
\end{abstract}
\maketitle
\section{Introduction}
The Lukacs theorem (\cite{Lukacs1955}) is one of the most celebrated characterizations of probability distributions. It states that \textit{if $X$ and $Y$ are independent, positive, non-degenerate random variables such that their sum and quotient are also independent then $X$ and $Y$ have gamma distributions with the same scale parameter}. This theorem has many generalizations. The most important in the multivariate setting were given in \cite{OlRu1962} and \cite{CaLe1996}, where the authors extended characterization to matrix and symmetric cones variate distributions, respectively. There is no unique way of defining the quotient of elements of the cone and in these papers the authors have considered very general form $U=[w(X+Y)]^{-1}X [w^T(X+Y)]^{-1}$, where $w$ is the so called division algorithm, that is, $w(\ab)w^T(\ab)=\ab$ for any element $\ab$ of the cone. The drawback of their extension was that the additional strong assumptions of invariance of the distribution of the ``quotient'' under a group of automorphisms was imposed. To avoid this assumption Bobecka and Weso\l{}owski \cite{BW2002} have developed another approach based on densities of $X$ and $Y$. Assuming existence of strictly positive, twice differentiable densities on the cone of positive definite symmetric matrices they proved a characterization of Wishart distribution for $U=(X+Y)^{-1/2}X(X+Y)^{-1/2}$, where $\ab^{1/2}$ denotes the unique symmetric root of a positive definite matrix $\ab$. Among all division algorithms the choice $w(\ab)=\ab^{1/2}$ is possibly the most natural one. 

Exploiting the same approach, with the same technical assumptions on densities it was proven in \cite{HaLaZi2008} that the independence of $X+Y$ and the quotient defined through the Cholesky decomposition, i.e. $U=[W(X+Y)]^{-1}X[W^T(X+Y)]^{-1}$, where $W(X+Y)$ is an upper triangular matrix in the decomposition of $X+Y$, characterizes a wider family of distributions called Riesz-Wishart. This fact shows that the invariance property assumed in \cite{OlRu1962} and \cite{CaLe1996} is not of technical nature only. Analogous results for homogeneous cones were obtained in \cite{Bout2009} and \cite{Hassairi2}.

In this paper we attempt to weaken the smoothness assumptions in the \cite{BW2002} version of the Lukacs theorem. OOur approach is based on the new solution of the Olkin–Baker equation and is very different from theirs. We succeeded in proving the characterization theorem assuming the continuity of densities only. This was possible thanks to an approach developed by Moln{\'a}r \cite{Molnar2006}. He proposed a method of solving functional equation for functions of matrix arguments consisting of two steps. First, an equation is solved for mutually orthogonal idempotents and afterwards a beautiful connection with Gleason's theorem is exploited.

This paper is organized as follows. We start in the next section with basic definitions and theorems regarding analysis on symmetric cones and a short introduction to Gleason's theorem. The statement and proof of the main result are given in Section \ref{secLUK}. Section \ref{secFUN} is devoted to solving two functional equations and is the technical core of the paper. It should be stressed, however, that we limit our considerations to symmetric cones of rank strictly greater than $2$ and exclude the octonion cone, due to the use of Gleason's theorem. 

\section{Preliminaries}
In this section we give a short introduction to the theory of symmetric cones and quantum logic. For further details refer, respectively, to \cite{FaKo1994} and \cite{Dvu1993}. 

A \textit{Euclidean Jordan algebra} is a Euclidean space $\En$ (endowed with scalar product denoted $\scalar{\xx,\yy}$) equipped with a bilinear mapping (product)
\begin{align*}
\En\times\En \ni \left(\xx,\yy\right)\mapsto \xx\yy\in\En
\end{align*}
and a neutral element $\ee$ in $\En$ such that for all $\xx$, $\yy$, $\zz$ in $\En$:
\begin{itemize}
	\item $\xx\yy=\yy\xx$, 
	\item $\xx(\xx^2\yy)=\xx^2(\xx\yy)$,
	\item $\xx\ee=\xx$,
	\item $\scalar{\xx,\yy\zz}=\scalar{\xx\yy,\zz}$.
\end{itemize}
For $\xx\in\En$ let $\LL(\xx)\colon \En\to\En$ be linear map defined by
\begin{align*}
\LL(\xx)\yy=\xx\yy,
\end{align*}
and define 
\begin{align*}
\PP(\xx)=2\LL^2(\xx)-\LL\left(\xx^2\right).
\end{align*} 
The map $\PP\colon \En\mapsto End(\En)$ is called the \emph{quadratic representation} of $\En$.

An element $\xx$ is said to be \emph{invertible} if there exists an element $\yy$ in $\En$ such that $\LL(\xx)\yy=\ee$. Then $\yy$ is called the \emph{inverse of} $\xx$ and is denoted by $\yy=\xx^{-1}$. Note that the inverse of $\xx$ is unique. It can be shown that $\xx$ is invertible if and only if $\PP(\xx)$ is invertible and in this case $\left(\PP(\xx)\right)^{-1} =\PP\left(\xx^{-1}\right)$.

Euclidean Jordan algebra $\En$ is said to be \emph{simple} if it is not a \mbox{Cartesian} product of two Euclidean Jordan algebras of positive dimensions. Up to linear isomorphism there are only five kinds of Euclidean simple Jordan algebras. Let $\mathbb{K}$ denote either the real numbers $\RR$, the complex ones $\mathbb{C}$, quaternions $\mathbb{H}$ or the octonions $\mathbb{O}$, and write $S_r(\mathbb{K})$ for the space of $r\times r$ Hermitian matrices valued in $\mathbb{K}$, endowed with the Euclidean structure $\scalar{\xx,\yy}=\Trace(\xx\cdot\bar{\yy})$ and with the Jordan product
\begin{align}\label{defL}
\xx\yy=\tfrac{1}{2}(\xx\cdot\yy+\yy\cdot\xx),
\end{align}
where $\xx\cdot\yy$ denotes the ordinary product of matrices and $\bar{\yy}$ is the conjugate of $\yy$. Then $S_r(\RR)$, $r\geq 1$, $S_r(\mathbb{C})$, $r\geq 2$, $S_r(\mathbb{H})$, $r\geq 2$, and the exceptional $S_3(\mathbb{O})$ are the first four kinds of Euclidean simple Jordan algebras. Note that in this case 
\begin{align}\label{defP}
\PP(\yy)\xx=\yy\cdot\xx\cdot\yy.
\end{align}
The fifth kind is the Euclidean space $\RR^{n+1}$, $n\geq 2$, with Jordan product
\begin{align}\label{scL}\begin{split}
\left(x_0,x_1,\dots, x_n\right)\left(y_0,y_1,\dots,y_n\right) =\left(\sum_{i=0}^n x_i y_i,x_0y_1+y_0x_1,\dots,x_0y_n+y_0x_n\right).
\end{split}
\end{align}

To each Euclidean simple Jordan algebra one can attach the set of Jordan squares
\begin{align*}
\bar{\VV}=\left\{\xx\in\En\colon\mbox{ there exists }\yy\mbox{ in }\En\mbox{ such that }\xx=\yy^2 \right\}.
\end{align*}
The interior $\VV$ is a symmetric cone.
Moreover $\VV$ is \emph{irreducible}, i.e. it is not the Cartesian product of two convex cones. One can prove that an open convex cone is symmetric and irreducible if and only if it is the cone $\VV$ of some Euclidean simple Jordan algebra. Each simple Jordan algebra corresponds to a symmetric cone, hence there exist up to linear isomorphism also only five kinds of symmetric cones. The cone corresponding to the Euclidean Jordan algebra $\RR^{n+1}$ equipped with Jordan product \eqref{scL} is called the Lorentz cone. 

We will now introduce a very useful decomposition in $\En$, called \emph{spectral decomposition}. An element $\cc\in\En$ is said to be a \emph{primitive idempotent} if $\cc\cc=\cc\neq 0$ and if $\cc$ is not a sum of two non-null idempotents. A \emph{complete system of primitive orthogonal idempotents} is a set $\left\{\cc_1,\dots,\cc_r\right\}$ such that
\begin{align*}
\sum_{i=1}^r \cc_i=\ee\quad\mbox{and}\quad\cc_i\cc_j=\delta_{ij}\cc_i\quad\mbox{for } 1\leq i<j\leq r.
\end{align*}
The size $r$ of such system is a constant called the \emph{rank} of $\En$. Any element $\xx$ of a Euclidean simple Jordan algebra can be written as $\xx=\sum_{i=1}^r\lambda_i\cc_i$ for some complete $\left\{\cc_1,\dots,\cc_r\right\}$ system of primitive orthogonal idempotents. The real numbers $\lambda_i$, $i=1,\dots,r$ are the \emph{eigenvalues} of $\xx$. One can then define \emph{trace} and \emph{determinant} of $\xx$ by, respectively, \mbox{$\tr(\xx)=\sum_{i=1}^r\lambda_i$} and \mbox{$\det(\xx)=\prod_{i=1}^r\lambda_i$}. An element $\xx\in\En$ belongs to $\VV$ if and only if all its eigenvalues are strictly positive. 
The \emph{logarithm} $\ab$ of a given element $\textbf{b}\in\VV$ is defined by $\exp(\ab)=\textbf{b}$ and we denote it by $\ab=\log\textbf{b}$. For $\VV\ni\xx=\sum_{i=1}^r \lambda_i\cc_i$ we have $\log\xx=\sum_{i=1}^r\cc_i\log\lambda_i$. 
 
The rank $r$ and $\dim\VV$ are connected through relation
\begin{align*}
\dim\VV=r+\frac{d r(r-1)}{2},
\end{align*}
where $d$ is an integer called the \emph{Peirce constant}. 

The Wishart distribution $\gamma_{p,\ab}$ in $\bar{\VV}$ is defined for any $\ab\in\VV$ and any $p$ in the set
\begin{align*}
\{0,d/2,d,\ldots,d(r-1)/2\}\cup(d(r-1)/2,\infty)
\end{align*}
by its Laplace transform
\begin{align*}
\int_{\bar{\VV}} \exp(-\scalar{\ttt,\yy})\gamma_{p,\ab}(d\yy)=\left(\det\left(\ee+\ttt \ab^{-1}\right)\right)^{-p}
\end{align*}
for any $\ttt+\ab\in\VV$. If $p>\dim\VV/r-1$ then $\gamma_{p,\ab}$ is absolutely continuous with respect to the Lebesgue measure and has density
\begin{align*}
\gamma_{p,\ab}(d\yy)=\frac{(\det(\ab))^p}{\Gamma_\VV(p)} (\det(\yy))^{p-\dim\VV/r}\exp\left(-\scalar{\ab,\yy}\right)I_\VV(\yy)\,d\yy,
\end{align*}
where $\Gamma_\VV$ is the Gamma function of symmetric cone $\VV$ (see, for instance, \cite{FaKo1994}, p. 122).

One of the crucial elements in the proof of the main theorem is the use of Gleason's theorem which originated from quantum logic theory. Now we introduce necessary basics.

Let $L(H)$ be the set of all closed subspaces of a real, complex or left quaternionic Hilbert space $H$ of finite dimension. We define \emph{charge} to be a mapping $m\colon L(H)\to\RR\cup\left\{-\infty\right\}\cup\left\{\infty\right\}$ such that
\begin{align*}
m(0)&=0, \\
m\left(\bigcup_{t\in T} a_t\right)&=\sum_{t\in T} m(a_t),
\end{align*}
for any system of mutually orthogonal subspaces $\left\{a_t\right\}_{t\in T}$ from $L(H)$ with any finite index set $T$.

We denote by $P_1(H)\subset L(H)$ the set of all one-dimensional subspaces of $H$. We say that charge $m$ is \emph{$P_1(H)$-semibounded} if 
\begin{align*}
\inf\left\{m(M)\colon M\in P_1(H)\right\}>-\infty.
\end{align*}
\begin{thm}[Gleason's theorem] \label{gT}
For every $P_1(H)$-semibounded charge $m$ on $L(H)$, $3\leq\dim H<\infty$, there exists a unique Hermitian operator $\ttt$ on $H$ such that 
\begin{align*}
m(M)=\Trace(\ttt\cdot\pb_M),\quad M\in L(H),
\end{align*}
where $\pb_M$ is the orthogonal projection onto $M$ and $\mathrm{Trace}$ is the trace of a linear operator on $H$.
\end{thm}

For a fuller description of above mentioned theory we refer to \cite{Dvu1993} (real and complex case) and \cite{Varad1985} (quaternionic case).

\section{Olkin--Baker Functional Equation}\label{secFUN}

In this section we solve the functional equation derived from the condition of independence of corresponding random variables. For details see Section \ref{secLUK}. The following theorem is of independent interest in the functional equations theory and is the technical core of the paper.

\begin{mainthm}[Olkin--Baker equation on symmetric cones]\label{T3}
Let $a$, $b$, $c$ and $d$ be real continuous functions on a non octonion symmetric cone $\VV$ of rank $r\neq 2$. Assume 
\begin{align}\label{czt}
a(\xx)+b(\yy)=c(\xx+\yy)+d\left(\PP\left((\xx+\yy)^{-1/2}\right)\xx\right),\qquad (\xx,\yy)\in \VV^2.
\end{align}
Then there exist constants $k_1,k_2\in\RR$, $\Lambda\in\En$, $C_i\in\RR$, $i\in\{1,2,3,4\}$ such that 
\begin{align}\label{gwiazdka}\begin{split}
a(\xx)&=\scalar{\Lambda,\xx}+k_1\log\det(\xx)+C_1,\\
b(\xx)&=\scalar{\Lambda,\xx}+k_2\log\det(\xx)+C_2,\\
c(\xx)&=\scalar{\Lambda,\xx}+(k_1+k_2)\log\det(\xx)+C_3,\\
d(\ub)&=k_1\log\det(\ub)+k_2\log\det(\ee-\ub)+C_4,
\end{split}\end{align}
for all $\xx\in\VV$ and $\ub\in\DD:=\left\{\zz\in\VV\colon\ee-\zz\in\VV \right\}$.
\end{mainthm}
The problem of solving
\begin{align}\label{OB}
f(x)g(y)=p(x+y)q(x/y),\quad(x,y)\in(0,\infty)^2
\end{align}
for unknown positive functions $f$, $g$, $p$ and $q$ was first posed in \cite{Olkin3}. Its general solution was given in \cite{Baker1976}, and later analyzed in \cite{Lajko1979} using a different approach. Recently, in \cite{Mesz2010} and \cite{LajMes12} the equation \eqref{OB} was solved assuming that it is satisfied almost everywhere on $(0,\infty)^2$ for measurable functions which are non-negative on its domain or positive on some sets of positive Lebesgue measure, respectively. Finally, a new derivation of solution to \eqref{OB}, when the equation holds almost everywhere on $(0,\infty)^2$ and no regularity assumptions on unknown positive functions are imposed, was given in \cite{WES4}. The equation \eqref{czt} is an adaptation of \eqref{OB} (after taking logarithm) to the symmetric cone case.

The proof is divided into two propositions and two lemmas.  In Proposition \ref{T1} we show that each of the functions $a$, $b$ and $c$ is a sum of the function in the expected form, as in \eqref{gwiazdka}, and a homogeneous one, satisfying the same functional equation \eqref{czt}. In Proposition \ref{T2} we show that these homogeneous functions are actually constant. 

We start with a simple lemma, which will be useful in the proof of the first proposition.
\begin{lem}{(Pexider functional equation on symmetric cones)}\label{lemma1}
Let $a$, $b$ and $c$ be measurable functions on a symmetric cone $\VV$ satisfying 
\begin{align}\label{lempex}
a(\xx)+b(\yy)=c(\xx+\yy),\qquad \forall\,(\xx,\yy)\in \VV^2.
\end{align}
Then there exist constants $\alpha, \beta\in\RR$ and $\lambda\in\En$ such that for all $\xx\in\VV$,
\begin{align}\label{abcdef}\begin{split}
a(\xx)&=\scalar{\lambda,\xx}+\alpha, \\
b(\xx)&=\scalar{\lambda,\xx}+\beta, \\
c(\xx)&=\scalar{\lambda,\xx}+\alpha +\beta.
\end{split}\end{align} 
\end{lem}
\begin{proof}
First we will show $c$ satisfies Jensen functional equation. For $\xx, \yy\in\VV$ we have
\begin{align*}
a(\xx)+b(\yy)=c(\xx+\yy)=a(\yy)+b(\xx).
\end{align*}
Hence, $a(\xx)-b(\xx)=a(\yy)-b(\yy)=const:=A_1$. Plugging $a(\xx)=b(\xx)+A_1$ into \eqref{lempex} for $\yy=\xx\in\VV$ we get $2b(\xx)=c(2\xx)-A_1$ and so
\begin{align*}
c(2\xx)+c(2\yy)=2c(\xx+\yy),
\end{align*}
i.e. $c$ is Jensen on symmetric cone $\VV$. Following the standard approach (see, for instance, \cite{GerKom2011}) we infer that there exists a constant $A_2$ such that $c(\xx)=f(\xx)+A_2$, where $f$ is additive on $\VV$, i.e.
\begin{align*}
f(\xx)+f(\yy)=f(\xx+\yy),\quad\xx,\yy\in\VV.
\end{align*}
We define an extension $\bar{f}$ of $f$ to the whole $\En$ as follows:
\begin{align*}
\bar{f}(\xx)=\begin{cases}
f(\xx) & \mbox{ for }\xx\in\VV, \\
f(\xx+t_\xx\ee)-f(t_\xx\ee)& \mbox{ for }\xx\notin\VV,
\end{cases}
\end{align*}
where $t_\xx=1 - \min_i \lambda_i$, $\lambda_i$ being the $i$-th eigenvalue of $\xx$. Observe that for $\xx\notin\VV$ element $\xx+t_\xx\ee$ belongs to $\VV$.  Indeed, $\xx+t_\xx\ee$ has all its eigenvalues positive, actually greater than $1$. Note that if $\xx\notin\VV$ then at least one of its eigenvalues is non-positive.
Since  $t_\xx>0$ we have $t_\xx\ee\in\VV$.

It can be easily verified that $\bar{f}$ is additive on $\En$. Take for example $\xx\in\VV$, $\yy\notin\VV$ and suppose $\xx+\yy\notin\VV$. Then we have
\begin{align*}
L & =\bar{f}(\xx)+\bar{f}(\yy)=f(\xx)+f(\yy+t_\yy\ee)-f(t_\yy\ee)\\
R & =\bar{f}(\xx+\yy)=f(\xx+\yy+t_{\xx+\yy}\ee)-f(t_{\xx+\yy}\ee)
\end{align*}
Since 
\begin{align*}
f(\xx)+f(\yy+t_\yy\ee)+f(t_{\xx+\yy}\ee) = f(\xx+\yy+t_\yy\ee+t_{\xx+\yy}\ee) = f(\xx+\yy+t_{\xx+\yy}\ee)+f(t_\yy\ee)
\end{align*}
we see that $L=R$. Hence, we know that there exists a constant $\lambda\in\En$ such that
\begin{align*}\bar{f}(\xx)=\scalar{\lambda,\xx}\end{align*}
for all $\xx\in\En$. Since $f(\xx)=\bar{f}(\xx)$ on $\VV$ the proof is complete.
One gets the form of \eqref{abcdef} for $A_1=\alpha-\beta$ and $A_2=\alpha+\beta$.
\end{proof}

\begin{prop} \label{T1}
Let $a$, $b$, $c$ and $d$ be real continuous functions on a symmetric cone $\VV$ of rank $r$. Assume 
\begin{align}\label{main}
a(\xx)+b(\yy)=c(\xx+\yy)+d\left(\PP\left((\xx+\yy)^{-1/2}\right)\xx\right),\qquad (\xx,\yy)\in \VV^2.
\end{align}
Then there exist constants $k_1,k_2\in\RR$ and $\Lambda\in\En$ such that for all $\xx\in\VV$ and $\ub\in\DD$,
\begin{align*}
a(\xx)&=\scalar{\Lambda,\xx}+k_1\log\det(\xx)+e(\xx),\\
b(\xx)&=\scalar{\Lambda,\xx}+k_2\log\det(\xx)+f(\xx),\\
c(\xx)&=\scalar{\Lambda,\xx}+(k_1+k_2)\log\det(\xx)+g(\xx),\\
d(\ub)&=k_1\log\det(\ub)+k_2\log\det(\ee-\ub)+h(\ub),
\end{align*}
where $e$, $f$, $g$ and $h$ are continuous functions satisfying 
\begin{align}
e(\xx)+f(\yy)=g(\xx+\yy)+h\left(\PP\left((\xx+\yy)^{-1/2}\right)\xx\right) \label{qwerty} \\
\intertext{for $(\xx,\yy)\in\VV^2$ and} 
e(s\xx)=e(\xx),\quad f(s\xx)=f(\xx),\quad g(s\xx)=g(\xx)  \label{prop}
\end{align}
for any $s\in(0,\infty)$ and $\xx\in\VV$.
\end{prop}
\begin{proof} The following proof adapts the argument given in \cite{WES4}, where the analogous result on $(0,\infty)$ was analyzed, to the symmetric cone setting.

For any $s>0$ and $(\xx,\yy)\in \VV^2$ we get
\begin{align}\label{mainr}
a(s\xx)+b(s\yy)=c(s(\xx+\yy))+d\left(\PP\left((\xx+\yy)^{-1/2}\right)\xx\right).
\end{align}
Substracting now \eqref{main} from \eqref{mainr} for any $s>0$ we arrive at the additive Pexider equation on symmetric cone $\VV$
\begin{align*}
a_s(\xx)+b_s(\yy)=c_s(\xx+\yy),\qquad (\xx,\yy)\in \VV^2,
\end{align*}
where  $a_s$, $b_s$ and $c_s$ are functions defined by $a_s(\xx):=a(s\xx)-a(\xx)$, $b_s(\xx):=b(s\xx)-b(\xx)$ and $c_s(\xx):=c(s\xx)-c(\xx)$.

Due to continuity of $a$, $b$ and $c$ and Lemma \ref{lemma1} it follows that for any $s>0$ there exist constants $\lambda(s)\in\En$, $\alpha(s)\in\RR$ and $\beta(s)\in\RR$ such that for any $\xx\in\VV$, 
\begin{align*}
a_s(\xx) &= \scalar{\lambda(s),\xx}+\alpha(s),\\
b_s(\xx) &= \scalar{\lambda(s),\xx}+\beta(s),\\
c_s(\xx) &= \scalar{\lambda(s),\xx}+\alpha(s)+\beta(s).
\end{align*}
By the definition of $a_s$ and the above observation it follows that for any $(s,t)\in(0,\infty)^2$ and $\zz\in\VV$ 
\begin{align*}
a_{st}(\zz)=a_t(s\zz)+a_s(\zz).
\end{align*}
Hence, 
\begin{align}\label{naQ}
\scalar{\lambda(st),\zz}+\alpha(st)=\scalar{\lambda(t),s\zz}+\alpha(t)+\scalar{\lambda(s),\zz}+\alpha(s).
\end{align}
Since \eqref{naQ} holds for any $\zz\in\VV$ we see that $\alpha(st)=\alpha(s)+\alpha(t)$ for all $(s,t)\in(0,\infty)^2$. That is $\alpha(s)=k_1\log\,s$ for $s\in(0,\infty)$, where $k_1$ is a real constant. 

On the other hand 
\begin{align}\label{trr}
\scalar{\lambda(st),\zz}=\scalar{\lambda(s),\zz}+\scalar{\lambda(t),s\zz}=\scalar{\lambda(t),\zz}+\scalar{\lambda(s),t\zz}
\end{align}
since one can interchange $s$ and $t$ on the left hand side.
Putting $s=2$ and denoting $\Lambda=\lambda(2)$ we obtain
 \begin{align*}
 \scalar{\lambda(t),\zz}=\scalar{\Lambda,\zz}(t-1)
 \end{align*}
 for $t>0$ and $\zz\in \VV$. 
 It then follows that for all $s\in(0,\infty)$ and $\zz\in \VV$, 
 \begin{align}\label{asz}
a_s(\zz)=a(s\zz)-a(\zz)=\scalar{\Lambda,\zz}(s-1)+k_1\log\,s.
\end{align}
Note that $\det(\alpha\xx)=\alpha^r\det(\xx)$. Define a new function $e$ by
\begin{align*}
a(\xx)=e(\xx)+\scalar{\Lambda,\xx}+\tfrac{k_1}{r}\log\det(\xx).
\end{align*}
By \eqref{asz} we obtain $e(s\xx)=e(\xx)$ for $s>0$ and $\xx\in \VV$. 

An analogous derivation shows that there exist functions $f$ and $g$ such that $f(s\xx)=f(\xx)$, $g(s\xx)=g(\xx)$ for $s>0$ and $\xx\in\VV$, and
\begin{align*}
b(\xx)&=\scalar{\Lambda,\xx}+\tfrac{k_2}{r}\log\det(\xx)+f(\xx),\\
c(\xx)&=\scalar{\Lambda,\xx}+\tfrac{k_1+k_2}{r}\log\det(\xx)+g(\xx),
\end{align*}
for $\xx\in\VV$. 

The definition of function $h$ completes the proof:
\begin{align*}
h(\xx)=d(\xx)-\tfrac{k_1}{r}\log\det(\xx)-\tfrac{k_2}{r}\log\det(\ee-\xx).
\end{align*}
In order to show \eqref{qwerty} the Cauchy theorem on the determinant of the product $\det(\PP(\xx)\yy)=\det(\xx)^2\det(\yy)$ is used.
\end{proof}

The aim of the following proposition is to show that functions $e$, $f$, $g$ and $h$ obtained in the previous proposition are actually constant. As mentioned before, we limit our considerations to symmetric cones of rank $r>2$ except octonion cone; this allows us to use Gleason's theorem. 

We start with a crucial lemma about the form of orthogonally additive function on the set of idempotents. Since every symmetric cone of rank $r>2$ is isomorphic to $r\times r$ Hermitian positive semidefinite matrices over $\KK$, we will identify trace $\tr$ with the trace of a linear operator, $\Trace$. 

\begin{lem}\label{Gleasonlemma}
Let $\En=S_r\left(\KK\right)$ where $\KK$ is either $\mathbb{R}$, $\mathbb{C}$ or $\mathbb{H}$ with Jordan product \eqref{defL} and $\VV$ be the cone of $\En$.
\begin{enumerate}

\item[(a)] Suppose that for a continuous function $f\colon \En\to\RR$ the equality 
\begin{align*}
f(\pb)+f(\qb)=f(\pb+\qb)
\end{align*} 
holds for all mutually orthogonal idempotents $\pb$ and $\qb$ in $\En$. If $r>2$ then there exists a unique Hermitian operator $\ttt\in\En$ such that
\begin{align*}
f(\pb)=\Trace(\ttt\cdot\pb),\quad\forall\,\pb\in\En.
\end{align*}

\item[(b)] If for any $\xx, \yy\in\VV$,
\begin{align}\label{eqMol}
\Trace\left(\ttt\cdot\log\xx\right)+2\,\Trace\left(\ttt\cdot\log \yy\right) = \Trace\left(\ttt\cdot\log\left(\PP(\yy)\xx\right)\right)
\end{align}
then there exists a constant $\vartheta\in\RR$ such that $\ttt=\vartheta\ee$.
\end{enumerate}
\end{lem}
\begin{proof}
(a) $f$ is orthogonally additive $\mathbb{R}$-valued function on the set of all idempotents. The set of all idempotents on $\En$ is compact in the norm topology, hence $f$ is bounded on this set. Let $H$ be a $r$-dimensional Hilbert space over $\KK$. For any closed subspace $M\subset H$ define 
\begin{align*}
m(M)=f(\pb_M),
\end{align*}
where $\pb_M$ is the orthogonal projection onto $M$. Note that for any $M\in L(H)$ there exists a unique idempotent $\pb_M\in\En$, and any idempotent $\pb\in\En$ is the orthogonal projection onto a subspace of $H$. Hence, there is a one-to-one correspondence between $L(H)$ and the set of idempotents of $\En$.

It can be easily verified that $m$ is a charge on $L(H)$. From Gleason's theorem we conclude that there exists a Hermitian operator $\ttt$ such that
\begin{align*}
m(M)=f(\pb_M)=\Trace(\ttt\cdot\pb_M),\quad\forall M\subset L(H).
\end{align*}

(b) By \eqref{defP} the right hand side of equation \eqref{eqMol} can be written as $\Trace\left(\ttt\cdot\log\left(\yy\cdot\xx\cdot\yy\right)\right)$. 
Such equation was considered in \cite[(2)]{Molnar2006}, where the assertion was proven for real and complex positive definite matrices only. Moln{\'a}r's proof can be rewritten virtually unchanged in our case. Here we will repeat this argument for completeness.

Let us pick any idempotents $\pb$ and $\qb$. Set $\xx=\ee+t\pb$ and $\yy=\ee+t\qb$ for any $t>-1$. Easy computations shows that
\begin{align*}
\yy\cdot\xx\cdot\yy=\ee+t(2\qb+\pb)+t^2(\qb+\pb\cdot\qb+\qb\cdot\pb)+t^3(\qb\cdot\pb\cdot\qb).
\end{align*}
Then for suitable $t$ we can expand the operators $\log(\xx\cdot\yy\cdot\xx)$, $\log(\xx)$ and $\log(\yy)$ into power series of $t$ with operator coefficients according to the formula
\begin{align*}
\log(\ee+\ab)=\sum_{n=1}^\infty \frac{(-1)^{n+1}\ab^n}{n},\quad\scalar{\ab,\ab}<1.
\end{align*}
Equating the coefficients of $t^3$ on both sides of \eqref{eqMol}, after some calculations we arrive at
\begin{align*}
\Trace\left(\ttt\cdot\pb\cdot\qb\cdot\pb\right)=\Trace\left(\ttt\cdot\qb\cdot\pb\cdot\qb\right)
\end{align*}
for any idempotents $\pb$ and $\qb$. Let $H$ be a Hilbert space as in (a) and take unit vectors $u,v\in H$. Since $\En$ is a matrix Jordan algebra we may put $\pb=u\otimes u$ and $\qb=v\otimes v$, where $\otimes$ is the tensor product. Then we obtain
\begin{align*}
\Trace(\ttt\cdot u\otimes u)=\Trace(\ttt\cdot v\otimes v),
\end{align*}
provided that $u$ and $v$ are not orthogonal to each other.
This equality implies that there exists a constant $\vartheta$ such that $\Trace(\ttt\cdot u\otimes u)=\vartheta$ for any unit vector $u\in H$. Inserting $u=x/\sqrt{\Trace(x\otimes x)}$ we get
\begin{align*} 
\Trace(\ttt\cdot x\otimes x)=\Trace(\vartheta\cdot x\otimes x),
\end{align*}
for all $x\in H$, which gives $\ttt=\vartheta\ee$. Hence we get Moln{\'a}r's result for elements of $\VV$ also.
\end{proof}
\begin{prop}\label{T2}
Let $e$, $f$, $g$ and $h$ be continuous functions on non-octonion symmetric cone $\VV$ of rank $r\neq2$. 
Assume that \eqref{qwerty} and \eqref{prop} hold true.
Then functions $e$, $f$, $g$ and $h$ are real constants.
\end{prop}
\begin{proof}
The case $r=1$ is trivial, since $e(\xx)=e(s\xx)=e(1)$ for $s=\xx^{-1}$.

Assume that $r>2$.
Pick any idempotent $\pb$ (not necessarily primitive) on $\En$ and denote $\pbo=\ee-\pb$. Put $\xx=\alpha\pb+\pbo$ and $\yy=\beta\pb+\pbo$ for $(\alpha,\beta)\in(0,\infty)^2$. Since
\begin{align*}
(\xx+\yy)^{-1/2}=((\alpha+\beta)\pb+2\pbo)^{-1/2}=\tfrac{1}{\sqrt{\alpha+\beta}}\pb+\tfrac{1}{\sqrt{2}}\pbo
\end{align*}
we obtain
\begin{align*}
e(\alpha\pb+\pbo)+f(\beta\pb+\pbo)=g((\alpha+\beta)\pb+2\pbo)+h\left(\tfrac{\alpha}{\alpha+\beta}\pb+\tfrac{1}{2}\pbo\right).
\end{align*}
This is a one-dimensional version of the main equation \eqref{main} in functions of variables $\alpha$ and $\beta$, so we already know its complete solution:
\begin{align}\label{star}
\begin{cases}
e(\alpha\pb+\pbo)=\lambda(\pb)\alpha+\kappa_1(\pb)\log\,\alpha+C_1(\pb), \\
f(\alpha\pb+\pbo)=\lambda(\pb)\alpha+\kappa_2(\pb)\log\,\alpha+C_2(\pb), \\
g(\alpha\pb+2\pbo)=\lambda(\pb)\alpha+(\kappa_1(\pb)+\kappa_2(\pb))\log\,\alpha+C_3(\pb), \\
h(\gamma\pb+\tfrac{1}{2}\pbo)=\kappa_1(\pb)\log\,\gamma+\kappa_2(\pb)\log(1-\gamma)+C_4(\pb), \\
C_1(\pb)+C_2(\pb)=C_3(\pb)+C_4(\pb),
\end{cases}
\end{align}
for $\alpha>0$, $\gamma\in(0,1)$ and any idempotent $\pb\in\En$.

Observe that utilizing the homogeneity of $e$ and the above result we arrive at $e(\alpha\pb+\pbo)=e(\tfrac{1}{\alpha}\pbo+\pb)$ and thus
\begin{align*}
\lambda(\pb)\alpha+\kappa_1(\pb)\log\,\alpha+C_1(\pb)=\lambda(\pbo)\tfrac{1}{\alpha}+\kappa_1(\pbo)\log\,\tfrac{1}{\alpha}+C_1(\pbo)
\end{align*}
for all $\alpha>0$. From this we conclude that $\lambda(\pb)=0$, $\kappa_1(\pbo)=-\kappa_1(\pb)$ and $C_1(\pbo)=C_1(\pb)$. Analogous properties can be proved for $\kappa_2$ and $C_2$. Moreover, putting $\alpha=1$ in the first equation of \eqref{star} we get
\begin{align*}
e(\ee)=C_1(\pb)
\end{align*}
for all idempotents $\pb\in\En$.
Therefore $C_i(\pb)=C_i$, $i=1,2$, for any idempotent $\pb$.

Putting $\alpha=2$ in the third equation of \eqref{star} gives
\begin{align*}
g(2\ee)=(\kappa_1(\pb)+\kappa_2(\pb))\log\,2+C_3(\pb)=:C_3
\end{align*}
for all idempotents $\pb\in\En$. This results in 
\begin{align*}
C_3(\pb)=C_3-(\kappa_1(\pb)+\kappa_2(\pb))\log\,2
\end{align*}
and so
\begin{align*}
g(\alpha\pb+\pbo) =g(2\alpha\pb+2\pbo)= (\kappa_1(\pb)+\kappa_2(\pb))\log\,\alpha+C_3.
\end{align*}

We will derive now the formula for $h(x\pb+y\pbo)$ for $(x,y)\in(0,1)^2$. From \eqref{qwerty} we have
\begin{align*}
e(\alpha\pb+\pbo)&+f(\pb+\beta\pbo) =g\left((\alpha+1)\pb+(\beta+1)\pbo\right)+h\left(\frac{\alpha}{\alpha+1}\pb+\frac{1}{\beta+1}\pbo\right)
\end{align*}
for $\alpha, \beta>0$. Taking $x=\frac{\alpha}{\alpha+1}$ and $y=\frac{1}{\beta+1}$ one gets
\begin{align*}
h(x\pb+y\pbo)=\kappa_1(\pb)\log\frac{x}{y}+\kappa_2(\pb)\log\frac{1-x}{1-y}+C_1+C_2-C_3
\end{align*}
for $x,y\in(0,1)$.

Summarizing the above calculations, formulas for functions $e$, $f$, $g$ and $h$ simplify to
\begin{align*}
\begin{cases}
e(\alpha\pb+\pbo)=\kappa_1(\pb)\log\,\alpha+C_1, \\
f(\alpha\pb+\pbo)=\kappa_2(\pb)\log\,\alpha+C_2, \\
g(\alpha\pb+\pbo)=(\kappa_1(\pb)+\kappa_2(\pb))\log\,\alpha+C_3, \\
h(x\pb+y\pbo)=\kappa_1(\pb)\log\frac{x}{y}+\kappa_2(\pb)\log\frac{1-x}{1-y}+C_4, \\
C_1+C_2=C_3+C_4,
\end{cases}
\end{align*}
for $\alpha>0$, $x,y\in(0,1)$ and any idempotent $\pb\in\En$. Note that the above solution is also valid for cones of rank $r=2$. 

Next step is to consider two mutually orthogonal idempotents $\pb$ and $\qb$ such that $\pb+\qb\neq\ee$. Such choice is possible only if $r\geq3$, because there must exist at least three non-null idempotents. 

Now we will rewrite the main equation \eqref{qwerty} in variables $\vb=\xx+\yy$ and $\ub=\PP\left((\xx+\yy)^{-1/2}\right)\xx$:
\begin{align*}
e(\PP(\vb^{1/2})\ub)+f(\PP(\vb^{1/2})(\ee-\ub))=g(\vb)+h(\ub).
\end{align*}
Put $\vb=t\qb+\qbo$ and $\ub=\beta(t\pb+\pbo)$ for $t>0$ and $\beta, \beta t\in (0,1)$. Then 
\begin{align}\label{PQ1}\begin{split}
\xx&=\PP(\vb^{1/2})\ub=\beta( t(\pb+\qb)+(\pb+\qb)^\bot), \\
\yy&=\PP(\vb^{1/2})(\ee-\ub)=(1-\beta)\left(t\qb+\tfrac{1-\beta t}{1-\beta}\pb+(\pb+\qb)^\bot\right).
\end{split}
\end{align}
After some easy but tedious computations one gets
\begin{align}
f(\yy) & =f\left(t\qb+\tfrac{1-\beta t}{1-\beta}\pb+(\pb+\qb)^\bot\right) \label{PQ2}\\
 & =\left(\kappa_2(\qb)+\kappa_1(\qb) + \kappa_1(\pb)-\kappa_1(\pb+\qb)\right)\log\,t +\kappa_2(\pb)\log\frac{1-\beta t}{1-\beta}+C_2. \nonumber
\end{align}
Due to the non symmetry in \eqref{PQ2} we also have
\begin{align*}
f(\yy)=(\kappa_2(\pb)+\kappa_1(\qb)+\kappa_1(\pb)-\kappa_1(\pb+\qb))\log\frac{1-\beta t}{1-\beta}+\kappa_2(\qb)\log\,t+C_2
\end{align*}
for $t>0$ and $\beta, \beta t\in (0,1)$. Hence, $\kappa_1(\qb)+\kappa_1(\pb)=\kappa_1(\pb+\qb)$ for mutually orthogonal idempotents $\pb$ and $\qb$ on $\En$. By the continuity of $f$ we get the continuity of $\kappa_1$. By Lemma \ref{Gleasonlemma} (a) there exists self-adjoint linear operator $\ttt_1\in\En$ such that
\begin{align*}
\kappa_1(\pb)=\Trace(\ttt_1\cdot\pb)
\end{align*} 
for any idempotent $\pb$. Similarly for $\kappa_2$ (just replace $\ub$ by $\ee-\ub$ in the above considerations)
\begin{align*}
\kappa_2(\pb)=\Trace(\ttt_2\cdot\pb).
\end{align*} 
Note that $\Trace(\ttt_i)=0$, by the previously proven property $k_i(\pb)=-k_i(\pbo)$ for $i=1,2$.

The aim is to show that $\ttt_1=\ttt_2=0$. As a first step, we prove
\begin{align}\label{iterat}\begin{split}
e\left(\sum_{i=1}^k \lambda_i \pb_i\right) & =\sum_{i=1}^k \Trace(\ttt_1\cdot\pb_i)\log\lambda_i +C_1, \\
f\left(\sum_{i=1}^k \lambda_i \pb_i\right) & =\sum_{i=1}^k \Trace(\ttt_2\cdot\pb_i)\log\lambda_i +C_2,\\
g\left(\sum_{i=1}^k \lambda_i \pb_i\right) & =\sum_{i=1}^k \Trace((\ttt_1+\ttt_2)\cdot\pb_i)\log\lambda_i +C_3,\\
h\left(\sum_{i=1}^k \gamma_i \pb_i\right) & =\sum_{i=1}^k \left( \Trace(\ttt_1\cdot\pb_i)\log\gamma_i + \Trace(\ttt_2\cdot\pb_i)\log(1-\gamma_i)\right) +C_4,
\end{split}\end{align}
for any system $\left\{\pb_i\right\}_{i=1}^k$ of orthogonal idempotents such that $\sum_{i=1}^k \pb_i=\ee$, $k\leq r$, $\lambda_i>0$ and $\gamma_i\in(0,1)$ for $i=1,\dots,k$.
We use induction on $k$:
\begin{itemize}
	\item (Basis) For $k=3$ the formulas for $e$ and $f$ were shown previously. For $g$ and $h$ the formulas are obvious since
	\begin{align}\label{iter}\begin{split}
	g(\xx)&=g(2\xx)=e(\xx)+f(\xx)-h\left(\tfrac{1}{2}\ee\right),\\
	h(\xx)&=e(\xx)+f(\ee-\xx)-g(\xx).
\end{split}\end{align}
	\item (Inductive step) Suppose now the formulas \eqref{iterat} are correct for $k=n<r$. We will deduce their validity for $n+1$.
Put $\vb=t\pb_n+\pb_n^\bot$ and $\ub=\beta t \pb_1 + \sum_{i=2}^{n-1}\beta_i \pb_i + \beta (\pb_1+\dots+\pb_{n-1})^\bot$ for $t>0$, $\beta$, $\beta t$, $\beta_i$, $i=1,\dots, n-1$, different and all in $(0,1)$. Then we have the following analogue of \eqref{PQ1}
\begin{align*}
\PP(\vb^{1/2})\ub & =\beta t (\pb_1+\pb_n)+\sum_{i=2}^{n-1}\beta_i \pb_i + \beta (\pb_1+\dots+\pb_n)^\bot,\\
\PP(\vb^{1/2})(\ee-\ub) & =(1-\beta t)\pb_1+\sum_{i=2}^{n-1}(1-\beta_i)\pb_i+t(1-\beta)\pb_n   +(1-\beta)(\pb_1+\dots+\pb_n)^\bot.
\end{align*} 
Now the result follows from similar calculation as in \eqref{PQ2} and \eqref{iter}. 
\end{itemize}
We see now that above formulas for $e$, $f$, $g$ and $h$ can be rewritten in a simpler form:
\begin{align*}
\begin{cases}
e(\xx)=\Trace\left(\ttt_1\cdot\log\xx\right)+C_1, \\
f(\xx)=\Trace\left(\ttt_2\cdot\log\xx\right)+C_2, \\
g(\xx)=\Trace\left((\ttt_1+\ttt_2)\cdot\log\xx\right)+C_3, \\
h(\ub)=\Trace\left(\ttt_1\cdot\log\ub\right)+\Trace\left(\ttt_2\cdot\log(\ee-\ub)\right)+C_4, \\
C_1+C_2=C_3+C_4,
\end{cases}
\end{align*}
for $\xx\in\VV$ and $\ub\in\DD$. Plug them into \eqref{qwerty} for $(r\xx,\yy)$, $r>0$, $\xx, \yy\in\VV$. Since $\log\left(r\xx\right)=\log\left(r\ee\right)+\log\xx=\log\left(r\right)\ee+\log\xx$ for any $\xx\in\VV$, we get 
\begin{align*}
\Trace\left(\ttt_1\cdot\log\left(\xx\right)\right)&+\Trace\left(\ttt_2\cdot\log \yy\right)=\Trace((\ttt_1+\ttt_2)\cdot\log(r\xx+\yy))\\
&+\Trace\left(\ttt_1\cdot\log\left(\PP\left((r\xx+\yy)^{-1/2}\right)\xx\right)\right) \\
&+\Trace\left(\ttt_2\cdot\log\left(\PP\left((r\xx+\yy)^{-1/2}\right)\yy\right)\right).
\end{align*}
Letting $r\to 0$, after some computations we obtain, for $\xx, \yy\in\VV$,
\begin{align*}
\Trace\left(\ttt_1\cdot\log\left(\xx\right)\right)-\Trace\left(\ttt_1\cdot\log \yy\right) =\Trace\left(\ttt_1\cdot\log\left(\PP\left(\yy^{-1/2}\right)\xx\right)\right)
\end{align*}
and finally by putting $\yy^{-2}$ in place of $\yy$ we arrive at
\begin{align*}
\Trace\left(\ttt_1\cdot\log\xx\right)+2\,\Trace\left(\ttt_1\cdot\log \yy\right) = \Trace\left(\ttt_1\cdot\log\left(\PP\left(\yy\right)\xx\right)\right)
\end{align*}
for $\xx, \yy\in\VV$. From Lemma \ref{Gleasonlemma} (b) we deduce that there exists real constant $\vartheta_1$ such that $\ttt_1=\vartheta_1\ee$. Since $\Trace\ttt_1=0$ we get that $\vartheta_1=0$. A similar argument holds for $\ttt_2$ and thus the proof is complete.
\end{proof}

Note that Proposition \ref{T1} holds true with weaker assumptions, when equation is satisfied only almost everywhere for measurable, rather than continuous functions (see, eg. \cite{WES4} for $(0,\infty)$ case and \cite{Kolodz} for Lorentz cone case). The proof of Proposition \ref{T1} is valid for all symmetric cones including the Lorentz cone, while the proof of Proposition \ref{T2} is limited to cones of rank not equal to $2$ and the octonion cone case is not covered.

\begin{proof}[ Proof of Theorem~\ref{T3}]
Follows immediately from Propositions \ref{T1} and \ref{T2}.
\end{proof}

\section{Lukacs-Olkin-Rubin theorem with densities on symmetric cones}\label{secLUK}
The main result of the paper is the following theorem. We relax the smoothness assumptions on respective densities from twice differentiability in \cite{BW2002} to continuity on symmetric cones. Relaxing the regularity assumption in multivariate case to measurability or even removing the assumption of density, without invoking the invariance of the quotient as in \cite{OlRu1962} and \cite{CaLe1996}, remains a challenge.  

\begin{mainthm}[Lukacs-Olkin-Rubin theorem with densities on symmetric cones]
Let $X$ and $Y$ be independent rv's valued in a non-octonion symmetric cone $\VV$ of rank $r>2$ with strictly positive and continuous densities. Set $V=X+Y$ and $U=\PP\left((X+Y)^{-1/2}\right)X$. If $U$ and $V$ are independent then there exist $\ab\in\VV$ and $p_1,p_2>\dim\VV/r-1$ such that $X\sim \gamma_{p_1,\ab}$ and $Y\sim \gamma_{p_2,\ab}$.
\end{mainthm}
\begin{proof}

Let $\psi\colon \VV\times\VV\to\DD\times\VV$ be the bijection defined by 
\begin{align*}
\psi(\xx,\yy)=(\ub,\vb)=\left(\PP\left((\xx+\yy)^{-1/2}\right)\xx,\xx+\yy\right).
\end{align*}
We have $(U,V)=\psi(X,Y)$. Inverse mapping $\psi^{-1}$ is defined by 
\begin{align*}
(\xx,\yy)=\psi^{-1}(\ub,\vb)=\left(\PP\left(\vb^{1/2}\right)\ub,\PP\left(\vb^{1/2}\right)(\ee-\ub)\right).
\end{align*}
We are looking for the Jacobian of the map $\psi^{-1}$, that is, the determinant of the linear map
\begin{align*}
\begin{pmatrix}
d\ub\\
d\vb
\end{pmatrix}
\mapsto
\begin{pmatrix}
d\xx \\
d\yy 
\end{pmatrix}
=
\begin{pmatrix}
d\xx/d\ub & d\xx/d\vb \\
d\yy/d\ub & d\yy/d\vb
\end{pmatrix}
\begin{pmatrix}
d\ub\\
d\vb
\end{pmatrix}.
\end{align*}
We have
\begin{align*}
J=\left| 
\begin{array}{cc}
\PP(\vb^{1/2}) & d\xx/d\vb \\
-\PP(\vb^{1/2}) & Id_\VV-d\xx/d\vb
\end{array}
\right| 
=
\left|
\begin{array}{cc}
\PP(\vb^{1/2}) & d\xx/d\vb \\
0 & Id_\VV
\end{array}
\right| = \DDet(\PP(\vb^{1/2})),
\end{align*}
where $\DDet$ denotes the determinant in the space of endomorphisms on $\VV$. From \cite[Proposition III.4.2]{FaKo1994} we get
\begin{align*}
\DDet\left(\PP\left(\vb^{1/2}\right)\right)=(\det(\vb))^{(\dim\VV)/r}.
\end{align*}
Now we can find the joint density of $(U,V)$. Since $(X,Y)$ and $(U,V)$ have independent components, the following identity holds almost everywhere with respect to Lebesgue measure: 
\begin{align*}
(\det(\xx+\yy))^{\dim\VV/r}f_X(\xx)f_Y(\yy)=f_U\left(\PP\left((\xx+\yy)^{-1/2}\right) \xx\right)f_V(\xx+\yy),
\end{align*}
where $f_X$,$f_Y$,$f_U$ and $f_V$ denote densities of $X$, $Y$, $U$ and $V$, respectively. Since the respective densities are assumed to be continuous, the above equation holds for every $\xx,\yy\in\VV$. Taking logarithms of both sides of the above equation we get 
\begin{align}\label{lukacs}
a(\xx)+b(\yy)=c(\xx+\yy)+d\left(\PP\left((\xx+\yy)^{-1/2}\right)\xx\right),
\end{align}
where
\begin{align*}
a(\xx)&=\log\, f_X(\xx),\\
b(\xx)&=\log\, f_Y(\xx),\\
c(\xx)&=\log\, f_V(\xx)-\frac{\dim\VV}{r}\log\det(\xx),\\
d(\ub)&=\log\, f_U(\ub),
\end{align*}
for $\xx\in\VV$ and $\ub\in\DD$.

The conclusion follows now directly from Theorem \ref{T3}. Thus there exist constants $k_1,k_2\in\RR$, $\Lambda\in\En$, $C_i\in\RR$, $i\in\{1,2,3,4\}$ such that 
\begin{align*}
f_X(\xx)&=e^{a(\xx)}=e^{C_1} e^{\scalar{\Lambda,\xx}}(\det(\xx))^{k_1},\\
f_Y(\xx)&=e^{b(\xx)}=e^{C_2} e^{\scalar{\Lambda,\xx}}(\det(\xx))^{k_2},
\end{align*}
for all $\xx\in\VV$. Since $f_X$ and $f_Y$ are densities it follows that $\ab=-\Lambda\in\VV$, $k_i=p_i-(\dim\VV)/r>-1$ and $e^{C_i}=(\det(\ab))^{p_i}/\Gamma_\VV(p_i)$, $i=1,2$.
\end{proof}

\subsection*{Acknowledgement} The author thanks J. Weso{\l}owski for helpful comments and discussions.


\begin{thebibliography}{HD}






\normalsize
\baselineskip=17pt


\bibitem[Ba76]{Baker1976}
J.~A. Baker.
\newblock On the functional equation {$f(x)g(y)=p(x+y)q(x/y)$}.
\newblock \emph{Aeq. Math.}, 14\penalty0 (3):\penalty0 493--506, 1976.

\bibitem[BW02]{BW2002}
K.~Bobecka and J.~Weso{\l}owski.
\newblock The {L}ukacs-{O}lkin-{R}ubin theorem without invariance of the
  ``quotient''.
\newblock \emph{Studia Math.}, 152\penalty0 (2):\penalty0 147--160, 2002.

\bibitem[Bo09]{Bout2009}
I.~Boutouria.
\newblock Characterization of the {W}ishart distribution on homogeneous cones
  in the {B}obecka and {W}esolowski way.
\newblock \emph{Comm. Statist. Theory Methods}, 38\penalty0 (13-15):\penalty0
  2552--2566, 2009.

\bibitem[BHM11]{Hassairi2}
I.~Boutouria, A.~Hassairi, and H.~Massam.
\newblock Extension of the {O}lkin and {R}ubin {C}haracterization to the
  {W}ishart distribution on homogeneous cones.
\newblock \emph{Inf. Dim. Anal. Quant. Probab. Rel. Top.}, 14\penalty0
  (4):\penalty0 591--611, 2011.

\bibitem[CL96]{CaLe1996}
M.~Casalis and G.~Letac.
\newblock The {L}ukacs-{O}lkin-{R}ubin characterization of {W}ishart
  distributions on symmetric cones.
\newblock \emph{Ann. Statist.}, 24\penalty0 (2):\penalty0 763--786, 1996.

\bibitem[Dv93]{Dvu1993}
A.~Dvure{\v{c}}enskij.
\newblock \emph{Gleason's theorem and its applications}, volume~60 of
  \emph{Mathematics and its Applications}.
\newblock Kluwer Academic Publishers Group, Dordrecht, 1993.

\bibitem[FK94]{FaKo1994}
J.~Faraut and A.~Kor{\'a}nyi.
\newblock \emph{Analysis on symmetric cones}.
\newblock Oxford Mathematical Monographs. The Clarendon Press Oxford University
  Press, New York, 1994.
\newblock Oxford Science Publications.

\bibitem[GK11]{GerKom2011}
R.~Ger and Z.~Kominek.
\newblock An interplay between {J}ensen's and {P}exider's functional equations
  on semigroups.
\newblock \emph{Ann. Univ. Sci. Budapest. Sect. Comput.}, 35:\penalty0
  107--124, 2011.

\bibitem[GMW13]{WES4}
R.~Ger, J.~Misiewicz, and J.~Weso\l{}owski.
\newblock The {L}ukacs-{O}lkin-{R}ubin theorem and the {O}lkin-{B}aker equation.
\newblock \emph{J. Math. Anal. Appl.}, 399: 599--607, 2013.

\bibitem[HLZ08]{HaLaZi2008}
A.~Hassairi, S.~Lajmi, and R.~Zine.
\newblock A characterization of the {R}iesz probability distribution.
\newblock \emph{J. Theoret. Probab.}, 21\penalty0 (4):\penalty0 773--790, 2008.

\bibitem[Ko10]{Kolodz}
B.~Ko\l{}odziejek.
\newblock The {W}ishart distribution on the {L}orentz cone.
\newblock Master's thesis, Fac. Math. Infor. Sci., Warsaw Univ. Tech., 2010.
\newblock - in Polish.

\bibitem[La79]{Lajko1979}
K.~Lajk{\'o}.
\newblock Remark to a paper: ``{O}n the functional equation
  {$f(x)g(y)=p(x+y)q(x/y)$}'' by {J}. {A}. {B}aker.
\newblock \emph{Aeq. Math.}, 19 (2-3):227--231, 1979.

\bibitem[LM12]{LajMes12}
K.~Lajk\'{o} and F.~M\'{e}sz\'{a}ros.
\newblock Multiplicative type functional equations arising from
  characterization problems.
\newblock \emph{Aeq. Math.}, pages 1--10, 2012.

\bibitem[Lu55]{Lukacs1955}
E.~Lukacs.
\newblock A characterization of the gamma distribution.
\newblock \emph{Ann. Math. Statist.}, 26:\penalty0 319--324, 1955.

\bibitem[Me10]{Mesz2010}
F.~M{\'e}sz{\'a}ros.
\newblock A functional equation and its application to the characterization of
  gamma distributions.
\newblock \emph{Aeq. Math.}, 79\penalty0 (1-2):\penalty0 53--59, 2010.

\bibitem[Mo06]{Molnar2006}
L.~Moln{\'a}r.
\newblock A remark on the {K}ochen-{S}pecker theorem and some characterizations
  of the determinant on sets of {H}ermitian matrices.
\newblock \emph{Proc. Amer. Math. Soc.}, 134\penalty0 (10):\penalty0
  2839--2848, 2006.

\bibitem[Ol75]{Olkin3}
I.~Olkin.
\newblock Problem (p128).
\newblock \emph{Aeq. Math.}, 12:\penalty0 290--292, 1975.

\bibitem[OR62]{OlRu1962}
I.~Olkin and H.~Rubin.
\newblock A characterization of the {W}ishart distribution.
\newblock \emph{Ann. Math. Statist.}, 33:\penalty0 1272--1280, 1962.

\bibitem[Va85]{Varad1985}
V.~S. Varadarajan.
\newblock \emph{Geometry of quantum theory}.
\newblock Springer-Verlag, New York, second edition, 1985.

\end{thebibliography}
\end{document}